\documentclass[11pt]{article}
\usepackage{amsmath}
\usepackage{dsfont}
\usepackage{mathrsfs}
\usepackage{amsmath,amssymb}
\usepackage{amsfonts}
\usepackage{hyperref}
\usepackage{amsthm}
\usepackage{graphicx}
\usepackage{subfigure}
\usepackage{xcolor}
\renewcommand{\qed}{\hfill\small{$\square$}\normalsize}

\hfuzz=\maxdimen
\theoremstyle{definition}
\newtheorem{lemma}{Lemma}[section]
\newtheorem{definition}[lemma]{Definition}
\newtheorem{proposition}[lemma]{Proposition}
\newtheorem{theorem}[lemma]{Theorem}

\newtheorem{remark}{Remark}

\numberwithin{equation}{section}

\renewcommand{\qed}{\hfill\small{$\square$}\normalsize}

\DeclareFixedFont{\Acknowledgment}{OT1}{cmr}{bx}{n}{14pt}
\begin{document}

\title{A Note on Discrete Einstein Metrics}
\author{Huabin Ge, Jinlong Mei, Da Zhou}

\maketitle

\begin{abstract}
In this note, we prove that the space of all admissible piecewise linear metrics parameterized by the square of length on a triangulated manifold is a convex cone. We further study Regge's Einstein-Hilbert action and give a more reasonable definition of discrete Einstein metric than the former version in \cite{G}. Finally, we introduce a discrete Ricci flow for three dimensional triangulated manifolds, which is closely related to the existence of discrete Einstein metrics.
\end{abstract}

\section{The space of piecewise linear metrics}
Consider an $n$ dimensional compact manifold $M$ with a triangulation $\mathcal{T}$. The triangulation is written as $\mathcal{T}=\{\mathcal{T}_0,\mathcal{T}_1,\cdots,\mathcal{T}_n\}$, where $\mathcal{T}_i$ ($0\leq i\leq n$) represents the set of all $i$ dimensional simplices. A piecewise linear metric is a map $l:\mathcal{T}_1\rightarrow (0,+\infty)$ making each simplex an Euclidean simplex.

There are two disadvantages to think of $l$ as the analogue of smooth Riemannian metric tensor $g$. For one thing, we know that $\mathfrak{M}_{\mathcal{T}}$, the space of all admissible piecewise linear metrics, is not convex (although it is a simply connected open set). For another, the scaling property of $l$ is not good enough. If the smooth Riemannian metric tensor $g$ scales to $cg$ in the smooth manifold $M^n$, then the length $l(\gamma)$ of a curve $\gamma:[0,1]\rightarrow M$ scales to $\sqrt{c}l(\gamma)$.

If we take $l^2$ as the direct analogue of metric tensor $g$, both the above two disadvantages can be overcome. The idea of considering the square of $l$, not $l$ itself, as an analogue of smooth Riemannian metric tensor comes naturally from the former work by the first author and Xu \cite{GX1}, where the idea has been used for piecewise linear manifolds with circle or sphere packing metrics.
Firstly, we have
\begin{theorem}\label{Thm-converx-metric-spc}
For manifold $M^n$ with triangulation $\mathcal{T}$, denote $g_{ij}=l_{ij}^2$ for each adjacent edge $i\thicksim j$. Then $\mathfrak{M}_{\mathcal{T}}^2$, the space of all admissible piecewise linear metrics parameterized by $g_{ij}$, is a nonempty connected open convex cone.
\end{theorem}
\noindent
\begin{proof}
Rivin \cite{R} first observed this fact for a single simplex case. Gu \emph{et al} \cite{Gu} proved this fact for $n=2$ by direct calculation. The proof here follows from Rivin's idea. For an $n$-simplex $\Delta$ embedded in the Euclidean space, we label all vertices as $v_0, v_1, \cdots, v_n$ and all $\frac{n(n+1)}{2}$ edges as $l_{01}, \cdots, l_{n-1\,n}$. For brevity, let $n^*=\frac{n(n+1)}{2}$, then we need to show
\begin{center}
$\mathfrak{M}_{\Delta}^2=\big\{\left(l^2_{01}, \cdots, l^2_{n-1\,n}\right)\in\mathds{R}^{n^*}\big|\:l_{01}, \cdots, l_{n-1\,n}$ are edges of some Euclidean $n$-simplex$\big\}$
\end{center}
is convex. Construct a map from $\mathfrak{M}_{\Delta}^2$ to the set of all symmetric $n\times n$ matrices, which transforms $(l^2_{01}, \cdots,l^2_{n-1\,n})$ to
\begin{displaymath}
\frac{1}{2}\left(
\begin{array}{ccccc}
 2l_{01}^2& l_{01}^2+l_{02}^2-l_{12}^2 & l_{01}^2+l_{03}^2-l_{13} & \cdots  & l_{01}^2+l_{0n}^2-l_{1n}^2\\
  * & 2l_{02}^2 & \l_{02}^2+l_{03}^2-l_{23}^2 & \cdots & l_{02}^2+l_{0n}^2-l_{2n}^2 \\
  * &  * & 2l_{03}^2 & \cdots &  l_{03}^2+l_{0n}^2-l_{3n}^2 \\
  \vdots &\vdots&\vdots& \ddots &  \vdots  \\
  * & * & * &\cdots& 2l_{0n}^2
\end{array}
\right).
\end{displaymath}
The above matrix is the Gram matrix of $n$ linear independent vectors $\vec{01}, \vec{02},\cdots,\vec{0n}$ and hence is positive definite. Obviously, the map is injective and surjective. Note that the set of all positive definite $n\times n$ matrices is a nonempty open convex subset of $\mathds{R}^{n^*}$. Thus
$\mathfrak{M}_{\Delta}^2$ is also a nonempty open convex subset of $\mathds{R}^{n^*}$.

Next we prove the theorem for general triangulations. Assuming all edges are labeled in turn as $e_1, \cdots, e_m$, where $m=|\mathcal{T}_1|$. Then for any $n$-simplex $\Delta=(v_0,\cdots,v_n)$ with edges $e_{i_1},\ldots, e_{i_{n^*}}$, ($i_1,\cdots,i_{n^*}\in\{1, 2, \cdots, m\}$), denote
$$\widetilde{\mathfrak{M}}_{\Delta}^2=\big\{\big(\cdots,l^2_{i_1},\cdots,l^2_{i_2},\cdots,l^2_{i_{n^*}},\cdots\big)\big|\big(l^2_{i_1},\cdots, l^2_{i_{n^*}}\big)\in\mathfrak{M}_{\Delta}^2 \big\}=\mathfrak{M}_{\Delta}^2\times\mathds{R}^{m-{n^*}},$$
we have

$$\mathfrak{M}_{\mathcal{T}}^2=\mathop{\cap}\limits_{\Delta \in \mathcal{T}_n} \widetilde{\mathfrak{M}}_{\Delta}^2.$$
This implies that $\mathfrak{M}_{\mathcal{T}}^2$ is a nonempty connected open convex cone of $\mathds{R}^m$.\qed
\end{proof}

\section{An interpretation of Regge's Einstein-Hilbert action}
Now we deal with three dimensional case. We give an interpretation to three dimensional Regge's Einstein-Hilbert action. The idea here is nature when taking $l^2$ as the analog of Riemann metric tensor $g$. Given an Euclidean tetrahedron $\{i,j,k,l\}\in \mathcal{T}_3$, the dihedral angle at edge $\{i,j\}$ is denoted as $\beta_{ij,kl}$. The discrete Ricci curvature $R_{ij}$ at edge $\{i,j\}$ is
\begin{equation}
R_{ij}=2\pi-\sum_{\{i,j,k,l\}\in \mathcal{T}_3}\beta_{ij,kl},
\end{equation}
where the sum is taken over all tetrahedrons having $\{i,j\}$ as one of their edges (If this edge is at the boundary of the triangulation, then the discrete Ricci curvature should be $R_{ij}=\pi-\sum_{\{i,j,k,l\}\in \mathcal{T}_3}\beta_{ij,kl}$). The Regge's Einstein-Hilbert action \cite{Re,CG,Gn} in this case is
\begin{equation}
\mathcal{E}=\sum_{j\thicksim i}R_{ij}l_{ij}.
\end{equation}

Denote $\alpha_{i,jkl}$ be the solid angle at vertex $i$ in a single tetrahedron. Cooper and Rivin \cite{CR} once defined a combinatorial scalar curvature at vertex $i$ as
$$S_i^{CR}=4\pi-\sum_{\{i,j,k,l\}\in \mathcal{T}_3}\alpha_{i,jkl},$$
where the sum is taken over all tetrahedrons having $i$ as one of their vertices. This definition satisfies the following combinatorial equality at each vertex $i$
\begin{equation}
S_i^{CR}=\sum_{j\thicksim i}R_{ij},
\end{equation}
which can be proved by Euler's characteristic formula for spheres. However, Cooper and Rivin's curvature is scaling invariant, hence it does not perform as well as smooth scalar curvature when the metric scales. In the following, we denote $g^{ij}$ as the inverse of $g_{ij}$, \emph{i.e.} $g^{ij}=g_{ij}^{-1}$. For any function $f$ defined on all edges, we can take a combinatorial trace at vertex $i$ with respect to $g$ as $(tr_gf)_i=\sum_{j\thicksim i}g^{ij}f_{ij}$. Thus we may define the discrete scalar curvature as $S=tr_gR$,\emph{ i.e.}, the discrete scalar curvature at vertex $i$ is
\begin{equation}
S_i=(tr_gR)_i=\sum_{j\thicksim i}g^{ij}R_{ij}.
\end{equation}
We also define a combinatorial volume element at vertex $i$ as
\begin{equation}
V_i=\sum_{j\thicksim i}l_{ij}^3.
\end{equation}

The volume element and scalar curvature can be multiplied together by a combinatorial multiplication $\ast$, which is defined as
\begin{equation}
(S\ast V)_i=\sum_{j\thicksim i}g^{ij}R_{ij}l_{ij}^3=\sum_{j\thicksim i}R_{ij}l_{ij}.
\end{equation}
Taking $(S\ast V)_i$ as a combinatorial analogue of smooth $Sdvol$, we can get
\begin{equation}
\mathcal{E}=\frac{1}{2}\sum_i(S\ast V)_i=\sum_{j\thicksim i}R_{ij}l_{ij},
\end{equation}
which is a combinatorial analogue of smooth Einstein-Hilbert functional $\int_M Sdvol$.

\section{Discrete Einstein metric}
Let $V=\sum_iV_i$ be the discrete volume of the manifold $M$. It is a combinatorial analogue of smooth volume $\int_M dvol$. The normalized Einstein-Hilbert functional
$$\frac{\int_M Sdvol}{(\int_M dvol)^{1-\frac{2}{3}}}$$
plays an essential role in studying three dimensional Yamabe problem. This inspires us to consider a normalized Regge's Einstein-Hilbert action
\begin{equation}
Q(g)=\frac{\mathcal{E}}{V^{\frac{1}{3}}}=\frac{\sum\limits_{i\sim j} R_{ij}l_{ij}}{\Big(\sum\limits_{i\sim j} l^3_{ij}\Big)^{\frac{1}{3}}}.
\end{equation}

Since $R_{ij}$ and $g_{ij}=l^2_{ij}$ are somewhat like smooth Ricci curvature $Ric$ and smooth metric $g$ respectively, the Einstein metric $g$ with $Ric=\lambda g$ on smooth manifolds $M$ can be transformed to a discrete version, \emph{i.e.}, a combinatorial metric $g$ with $R=\lambda g$ on a triangulated manifold $(M^3, \mathcal{T})$. This fact inspires us to define the following discrete Einstein metric
\begin{definition}\label{Def-2orderDEM}
Given a compact manifold $M^3$ with triangulation $\mathcal{T}$. A piecewise linear metric $g$ is called a discrete Einstein metric if there exists $\lambda$ such that $R=\lambda g$.
\end{definition}

It is easy to see that, if $R=\lambda g$, or say $g$ is a discrete Einstein metric, then
\begin{equation}
\lambda=\frac{\mathcal{E}}{V}=\frac{\sum\limits_{i\sim j} R_{ij}l_{ij}}{\sum\limits_{i\sim j} l^3_{ij}}.
\end{equation}

\begin{theorem}
Given a manifold $M^3$ with triangulation $\mathcal{T}$, then a piecewise linear metric $g$ is a discrete Einstein metric if and only if it is a critical point of the normalized Regge's Einstein-Hilbert action.
\end{theorem}
\begin{proof}
The Schl$\ddot{a}$fli formula \cite{M} says
$$\sum_{i\sim j}l_{ij}dR_{ij}=0.$$
Hence we get $d\mathcal{E}=\sum\limits_{i\sim j}R_{ij}dl_{ij}$. Then
\begin{equation*}
\nabla_{g}Q=\frac{V^{-\frac{1}{3}}}{2}diag^{-\frac{1}{2}}\{g_1,\cdots,g_m\}(R-\lambda g),
\end{equation*}
which implies the conclusion. \qed
\end{proof}

We want to know how many discrete Einstein metrics there are for a fixed three manifold $M$ with triangulation. How to find them? Can we triangulate $M$ suitably, so as $M$ admits a discrete Einstein metric? We introduce a new topological-combinatorial invariant, which carries the information of the triangulation,
\begin{equation}
Y_{M, \mathcal{T}}=\inf_{g\in \mathfrak{M}_{\mathcal{T}}^2} Q(g)
\end{equation}
associated with a fixed triangulation $\mathcal{T}$ on a fixed manifold $M^3$. We also introduce a pure topological invariant $Y_M=\sup_\mathcal{T} Y_{M, \mathcal{T}}$ associated with a fixed manifolds $M^3$, where the supremum is taken from all triangulations of $M$. We hope the study of these two combinatorial and topological invariants will deepen the understanding of discrete Einstein metrics.

\section{Discrete Einstein metric of $\alpha$-order}\label{Remark-alpha metric}

Similar to the works by Ge and Xu \cite{GX1,GX2}, almost all procedures about discrete metric $g_i$ and discrete curvature $R_i$ can be generalized to $\alpha$ order. We can define $\alpha$-order discrete Einstein metric, \emph{i.e.} a metric $l$ satisfying $R=\lambda_{\alpha} l^{\alpha}$ ($\alpha\in\mathds{R}$).

The logic behind this procedure is to take $l^{\alpha}$ as a metric (of $\alpha$ order). From the viewpoint of Riemannian geometry, a piecewise flat metric is a singular Riemannian metric on $M^{3}$, which produces conical singularities at all vertices. For any $\alpha \in \mathds{R}$, a metric $g$ with conical singularity at a point can be expressed as
$$g(z)=e^{f(z)}\frac{dzd\bar{z}}{|z|^{2(1-\alpha)}}$$
locally. If letting $f(z)=-\ln(\alpha^2)$, then $g(z)=|dz^{\alpha}|^{2}$. Compared $l^{\alpha}$ with $|dz^{\alpha}|$, the $\alpha$-metric $l^{\alpha}$ could be considered as a discrete analogue of conical metric to some extent. Inspired by the essential role the conical metric plays in the study of smooth canonical or extremal metrics, we want to know what happens if we triangulates the manifold and evaluates the piecewise linear metrics of it. It seems that canonical or extremal metrics in the piecewise linear metric case should be a metric $l^{\alpha}$ parallels to the discrete Ricci curvature $R$, \emph{i.e.}, a discrete Einstein metric of $\alpha$-order.

It is easy to see that, if $l$ is a discrete Einstein metric of $\alpha$-order, then
\begin{equation}
\lambda_{\alpha}=\frac{\sum\limits_{i\sim j} R_{ij}l_{ij}}{\sum\limits_{i\sim j} l^{\alpha+1}_{ij}}.
\end{equation}
Furthermore, $\alpha$-order discrete Einstein metrics are critical points of the $\alpha$-order normalized Regge's Einstein-Hilbert action
\begin{equation}
Q_{\alpha}=\frac{\sum\limits_{i\sim j} R_{ij}l_{ij}}{\Big(\sum\limits_{i\sim j} l^{\alpha+1}_{ij}\Big)^\frac{1}{\alpha+1}}.
\end{equation}
By this definition, the discrete Einstein metric in Definition \ref{Def-2orderDEM} is in fact a $2$-order discrete Einstein metric, the discrete Einstein metric defined in \cite{G} is actually a $1$-order discrete Einstein metric, and the discrete Ricci flat metric is a $0$-order discrete Einstein metric.

\section{Discrete Ricci flow}
\begin{definition}
Given a manifold $M^3$ with triangulation $\mathcal{T}$, the discrete Ricci flow is
\begin{equation}\label{CRF}
\frac{dg_{ij}}{dt}=-2R_{ij}, \quad\mathrm{or}\quad \frac{dg}{dt}=-2R.
\end{equation}
\end{definition}

It is useful to consider the normalized discrete Ricci flow
\begin{equation}\label{NCRF}
\frac{dg_{ij}}{dt}=\frac{2}{3}rg_{ij}-2R_{ij}, \quad\mathrm{or}\quad \frac{dg}{dt}=\frac{2}{3}rg-2R,
\end{equation}
where $r=3\lambda=3\mathcal{E}/V$.

Discrete Ricci flow (\ref{NCRF}) takes the same form as the smooth Ricci flow $\frac{\partial g}{\partial t}=\frac{2}{n}rg-2Ric$ with $n=3$. It is easy to have
\begin{proposition}
Along the normalized discrete Ricci flow (\ref{NCRF}), $V$ is a constant, $\mathcal{E}$ is descending and bounded. Moreover, Ricci flow (\ref{NCRF}) is a negative gradient flow. \qed
\end{proposition}

\begin{theorem}
If the solution $g(t)$ of the discrete Ricci flow (\ref{NCRF}) converges to a non-degenerate metric $g(+\infty)$, then $g(+\infty)$ is a discrete Einstein metric.\qed
\end{theorem}

Recall that the Schl$\ddot{a}$fli formula \cite{M} implies $\sum\limits_{i=1}^ml_idR_i=0$ and hence $d\mathcal{E}=\sum\limits_{i=1}R_{i}dl_{i}$. Thus we have
$\nabla_l\mathcal{E}=(\frac{\partial \mathcal{E}}{\partial l_1},\cdots,\frac{\partial \mathcal{E}}{\partial l_m})^T=R$. Denote $L=\frac{\partial (R_{1},\cdots,R_{m})}{\partial(l_{1},\cdots,l_{m})}$, then $L=Hess_l\mathcal{E}$ is a symmetric matrix. However, generally $L$ is not positive definite \cite{G}. Denote
$$\Delta_{L}=\frac{\partial (R_{1},\cdots,R_{m})}{\partial(g_{1},\cdots,g_{m})},$$
then $\Delta_{L}$ could be thought of as a discrete version of Lichnerowicz Laplacian. Denote $\lambda_{inf}(\Delta_{L})$ as the smallest eigenvalue of $\Delta_{L}$, then we have

\begin{theorem}
Let $g^*$ be a discrete Einstein metric. If $\lambda_{inf}(\Delta_{L}^*)>\lambda^*$ and the initial metric $g(0)$ deviates from $g^*$ not so much, then the solution to discrete Ricci flow (\ref{NCRF}) exists for all $t\geq0$ and converges exponentially fast to the discrete Einstein metric $g^*$.
\end{theorem}
\begin{proof}
The proof of this theorem is similar to those appeared in \cite{G,GX2,GX3}.
We show that $g^*$ is an asymptotically stable point of the normalized discrete Ricci flow (\ref{NCRF}). Denote $\Gamma(g)=\frac{2}{3}rg-2R=2(\lambda g-R)$, then $g^*$ is a critical point of $\Gamma(g)$. We denote $\Sigma=diag\{l_1,\cdots,l_m\}$ and differentiate $\Gamma(g)$ at $g^*$,
\begin{equation}
\frac{1}{2}D_{g^*}\Gamma(g)=\lambda^* (I_{m}-\frac{l^{*2}l^{*T}}{V^*})-\Delta_{L}^*.
\end{equation}
If $\lambda_{inf}(\Delta_{L}^*)>\lambda^*$, then by similar methods in \cite{G,GX2,GX3}, we can show that all eigenvalues of the matrix $D_{g^*}\Gamma(g)$ are negative (up to scaling of metric $l^{\frac{3}{2}}$). The proof is completed. \qed\\[2pt]
\end{proof}

Similarly, for any $\alpha\in \mathds{R}-\{0,-1\}$, we can also define the normalized $\alpha$-order discrete Ricci flow as
\begin{equation}\label{NCRF-alpha-order}
\frac{dl_{ij}^{\alpha}}{dt}=2\lambda_{\alpha}l_{ij}^{\alpha}-2R_{ij},
\end{equation}
which can be written in vector form as
\begin{equation}
\frac{dl^{\alpha}}{dt}=2\lambda_{\alpha}l^{\alpha}-2R.\\[2pt]
\end{equation}

\begin{remark}
We have recently learned that R. Schrader got the same result (Theorem 3.1, \cite{Sc}) as our Theorem \ref{Thm-converx-metric-spc} in this paper.
R. Schrader \cite{Sc} also provided some interesting analogues of Ricci tensor and Ricci flow in the theory of piecewise linear spaces.\\[2pt]
\end{remark}

\section{An example: 16-cell triangulation}
Discrete Ricci flow provides an efficient way to find discrete Einstein metrics on triangulated manifolds. To see the benefit of this method, we take the 16-cell triangulation of $\mathbb{S}^3$ as an example. Here 16-cell is defined as a triangulation $\mathcal{T}_s$ of $\mathbb{S}^3$ like this: take $A_{1}=(1,0,0,0)$, $A_{2}=(-1,0,0,0)$, $B_{1}=(0,1,0,0)$, $B_{2}=(0,-1,0,0)$, $C_{1}=(0,0,1,0)$, $C_{2}=(0,0,-1,0)$, $D_{1}=(0,0,0,1)$, $D_{2}=(0,0,0,-1)$ as the vertices of $\mathcal{T}_s$;  $P_{i}Q_{j}(\{P,Q\}\in\{A,B,C,D\},i,j=1,2)$ as the edges of $\mathcal{T}_s$; $P_{i}Q_{j}R_{k}$ $(\{P,Q,R\}\subset\{A,B,C,D\},i,j,k=1,2)$ as the faces of $\mathcal{T}_s$; and regular tetrahedrons $A_{i}B_{j}C_{k}D_{l}(i,j,k,l=1,2)$ as the tetrahedrons of $\mathcal{T}_s$.

Now we consider a topological triangulation $\mathcal{T}$ of $\mathbb{S}^{3}$, which has the same combinatorial structure with $16$-cell $\mathcal{T}_s$. Obviously, $\mathcal{T}$ carries a trivial $\alpha$-order discrete Einstein metric for each $\alpha$, \emph{i.e.} a metric with the same value on all edges. We want to know whether there exist other non-trivial discrete Einstein metrics? Let $l_{P_iQ_j}$ be the lengths of $P_iQ_j$, and denote $g_{P_iQ_j}=l^2_{P_iQ_j}$. Consider the normalized discrete Ricci flow
\begin{equation}
\frac{dg_{P_iQ_j}}{dt}=2\lambda g_{P_iQ_j}-2R_{P_iQ_j}.\\[2pt]
\label{16cell}
\end{equation}
We now specialize the parameters in equation (\ref{16cell}):
\begin{enumerate}
  \item $\lambda=\frac{\sum_{\{P_iQ_j\}}R_{P_iQ_j}l_{P_iQ_j}}{\sum_{\{P_iQ_j\}}l^3_{P_iQ_j}}$, the sum is taken over all the twenty four edges in $\mathcal{T}$.
  \item To illustrate the calculation of $R_{P_iQ_j}$, we take $R_{A_1B_1}$ as an example. Note that
\begin{equation}
R_{A_1B_1}=2\pi-\beta_{A_1B_1, C_1D_1}-\beta_{A_1B_1, C_1D_2}-\beta_{A_1B_1, C_2D_1}-\beta_{A_1B_1, C_2D_2},
\end{equation}
where $\beta_{A_1B_1, C_iD_j}$ ($i,j=1,2$) are the dihedral angles at edge $A_1B_1$,
\begin{equation}
\sin\beta_{A_1B_1, C_iD_j}=\frac{3l_{A_1B_1} V_{A_1B_1C_iD_j}}{2 S_{A_1B_1C_i}S_{A_1B_1D_j}},
\end{equation}
\begin{equation}
\begin{aligned}
12^2V_{A_1B_1C_iD_j}^ 2
=& l^2_{A_1B_1}l^2_{C_iD_j}(l^2_{A_1C_i}+l^2_{A_1D_j}+l^2_{B_1C_i}+l^2_{B_1D_j}-l^2_{A_1B_1}-l^2_{C_iD_j})\\
& + l^2_{A_1D_j}l^2_{B_1C_i}(l^2_{A_1B_1}+l^2_{C_iD_j}+l^2_{A_1C_i}+l^2_{B_1D_j}-l^2_{A_1D_j}-l^2_{B_1C_i})\\
& + l^2_{A_1C_i}l^2_{B_1D_j}(l^2_{A_1B_1}+l^2_{C_iD_j}+l^2_{A_1D_j}+l^2_{B_1C_i}-l^2_{A_1C_i}-l^2_{B_1D_j})\\
& - l^2_{A_1B_1}l^2_{B_1C_i}l^2_{A_1C_i}-l^2_{A_1C_i}l^2_{A_1D_j}l^2_{C_iD_j}-l^2_{A_1B_1}l^2_{A_1D_j}l^2_{B_1D_j}\\
& -l^2_{B_1C_i}l^2_{C_iD_j}l^2_{B_1D_j},
\end{aligned}
\end{equation}
\begin{equation}
S_{A_1B_1C_i}=\sqrt{P_{A_1B_1C_i}(P_{A_1B_1C_i}-l_{A_1B_1})(P_{A_1B_1C_i}-l_{A_1C_i})(P_{A_1B_1C_i}-l_{B_1C_i})},
\end{equation}
\begin{equation}
S_{A_1B_1D_j}=\sqrt{P_{A_1B_1D_j}(P_{A_1B_1D_j}-l_{A_1B_1})(P_{A_1B_1D_j}-l_{A_1D_j})(P_{A_1B_1D_j}-l_{B_1D_j})},
\end{equation}
where $P_{A_1B_1C_i}=(l_{A_1B_1}+l_{A_1C_i}+l_{B_1C_i})/2$ and $P_{A_1B_1D_j}=(l_{A_1B_1}+l_{A_1D_j}+l_{B_1D_j})/2$.
Similarly, we can calculate all the other $R_{P_iQ_j}$.
\end{enumerate}
It is easy to see that $g_{P_iQ_j}(t)\equiv 1$ is a trivial solution of equation (\ref{16cell}). With the aid of \textsf{Matlab} software, we can solve equation (\ref{16cell}) numerically. The numerical calculations show that discrete Ricci flow equation (\ref{16cell}) may develop singularities for some initial values. However, if we choose the initial values appropriately, the solution $g(t)$ to equation (\ref{16cell}) can converge and hence deforms the metric $g(t)$ to some discrete Einstein metric $g^*$.
Write the solution $g(t)$ as a $1\times 24$ row vector with order $l^2_{A_{1}B_{1}}$, $l^2_{A_{1}B_{2}}$, $l^2_{A_{2}B_{1}}$, $l^2_{A_{2}B_{2}}$, $l^2_{A_{1}C_{1}},\cdots,l^2_{A_{1}D_{1}},\cdots,l^2_{B_{1}C_{1}},\cdots,l^2_{B_{1}D_{1}},\cdots,l^2_{C_{1}D_{1}},\cdots$. Denote $\textbf{1}_{n}$ as a $1\times n$ row vector with all entries being $1$, then we can give two specific discrete Einstein metrics as follows:
\begin{enumerate}
  \item  $g^*=(a\textbf{1}_{12}, b\textbf{1}_{12})$, where $a=10.0095$, $b=6.9633$.
  \item  $g^*=(c\textbf{1}_{8}, d\textbf{1}_{4}, e\textbf{1}_{4}, c\textbf{1}_{8})$, where $c=58.7223$, $d=64.9735$, $e=52.1413$.
\end{enumerate}

\noindent
\textbf{Acknowledgements}. The first author would like to show his greatest respect to Professor Gang Tian who brought him to this research area. The authors would also like to thank Xu Xu, Ma Shiguang, Shen Liangming, Zhang Shijin for many helpful discussions. This work is supported by NSFC grant (Nos.11501027 and 11401499),
Natural Science Foundation of Fujian Province of China (No. 2015J05016), and Fundamental Research Funds for the Central Universities (No. 20720140524).

\bibliographystyle{amsplain}

(Huabin Ge)
Department of Mathematics, Beijing Jiaotong University, Beijing 100044, P.R. China

E-mail: hbge@bjtu.edu.cn\\[2pt]

(Jinlong Mei)
School of Mathematical Sciences, Xiamen University, Xiamen 361005, P.R. China

E-mail: mjl948512922@outlook.com\\[2pt]

(Da Zhou)
School of Mathematical Sciences, Xiamen University, Xiamen 361005, P.R. China

E-mail: zhouda@xmu.edu.cn\\[2pt]

\end{document}